\documentclass{amsart}
\usepackage{amsfonts,amssymb,amsmath,amsthm}
\usepackage{url}
\usepackage{enumerate}

\usepackage{graphics}
\usepackage{graphicx}
\usepackage{epsfig}
\usepackage{amscd}
\usepackage{amssymb}
\usepackage{amsthm}
\usepackage{color}
\usepackage{extarrows}
\usepackage{amsmath}
\numberwithin{equation}{section}
\usepackage{mathrsfs}
\usepackage{hyperref}
\usepackage{bm}
\usepackage{float}
\usepackage{geometry}
\usepackage{tikz}
\usepackage{ytableau}
\usepackage{multirow}
\usepackage{verbatim}
\usepackage{geometry}

\usetikzlibrary{matrix}
\usetikzlibrary{arrows}
\usetikzlibrary{decorations.pathreplacing,decorations.markings}
\usetikzlibrary{matrix,arrows}
\usetikzlibrary{arrows,shapes,positioning}
\usetikzlibrary{decorations.markings}
\tikzstyle arrowstyle=[scale=1]
\tikzstyle directed=[postaction={decorate,decoration={markings,
    mark=at position .5 with {\arrow[arrowstyle]{stealth}}}}]
\tikzstyle reverse directed=[postaction={decorate,decoration={markings,
    mark=at position .65 with {\arrowreversed[arrowstyle]{stealth};}}}]

\newtheorem{thm}{Theorem}[section]
\newtheorem{lem}[thm]{Lemma}
\newtheorem{prop}[thm]{Proposition}
\newtheorem{cor}[thm]{Corollary}
\theoremstyle{definition}
\newtheorem{defn}[thm]{Definition}

\numberwithin{equation}{section}

\author{Ye Liu}
\address{
Department of Mathematics\\
Hokkaido University\\
North 10, West 8, Kita-ku, Sapporo, 060-0810, JAPAN}
\email{liu@math.sci.hokudai.ac.jp}

\begin{document}

\title[Chromatic functors and stable partitions]{On chromatic functors and stable partitions of graphs}

\begin{abstract}
The chromatic functor of a simple graph is a functorization of the chromatic polynomial. M. Yoshinaga showed in \cite{Yoshinaga2015} that two finite graphs have isomorphic chromatic functors if and only if they have the same chromatic polynomial. The key ingredient in the proof is the use of stable partitions of graphs. The latter is shown to be closely related to chromatic functors. In this note, we further investigate some interesting properties of chromatic functors associated to simple graphs using stable partitions. Our first result is the determination of the group of natural automorphisms of the chromatic functor, which is in general a larger group than the automorphism group of the graph. The second result is that the composition of the chromatic functor associated to a finite graph restricted to the category $\mathrm{FI}$ of finite sets and injections with the free functor into the category of complex vector spaces yields a consistent sequence of representations of symmetric groups which is representation stable in the sense of Church-Farb \cite{Church2013}.
\end{abstract}

\maketitle
\section{Introduction and definitions}
In the theory of graph coloring, we restrict ourselves to simple graphs, which are graphs with no loops or multiedges. By a graph $G=(V,E)$, we always mean an undirected simple graph $G$ with vertex set $V$ and edge set $E\subset 2^V$. A regular (vertex) $S$-coloring of $G$ with color set $S$ is a map $c:V\to S$ such that $c(v)\neq c(u)$ if $\{v,u\}\in E$ is an edge of $G$. A (regular) coloring with color set $[n]=\{1,2,\ldots,n\}$ is simply called a (regular) $n$-coloring. A (simple) graph $G$ is finite if $V$ is finite.

The chromatic polynomial of a finite graph $G$ is the polynomial $\chi(G,t)\in\mathbb{Z}[t]$ satisfying
\[
\chi(G,n)=\#\{c:V\to [n]\mid c(v)\neq c(u) \text{ if } \{v,u\}\in E\}
\]
for all $n>0$ (cf. \cite{Read1968}). The chromatic functor associated to a graph $G$ is introduced in \cite{Yoshinaga2015} as a functorization of $\chi(G,t)$.
\begin{defn}[Chromatic functor]
Let $G$ be a graph, define the chromatic functor associated to $G$ 
\[
\underline{\chi}(G,\bullet):\mathbf{Set}^{\mathbf{inj}}\to\mathbf{Set}^{\mathbf{inj}}
\]
between the category of sets and injections by setting
\[
\underline{\chi}(G,S):=\{c:V\to S\mid c(v)\neq c(u) \text{ if } \{v,u\}\in E\}
\]
for a set $S$ and the injection
\begin{align*}
\iota_*:=\underline{\chi}(G,\iota):\underline{\chi}(G,S)&\longrightarrow \underline{\chi}(G,T)\\
c&\longmapsto \iota\circ c
\end{align*}
induced by an injection $\iota:S\to T$.
\end{defn}
The main result of \cite{Yoshinaga2015} is the following theorem.
\begin{thm}[\cite{Yoshinaga2015}]\label{Yos}
Let $G_1, G_2$ be finite graphs, then the following are equivalent.

(i)~  $\chi(G_1,t)=\chi(G_2,t)$;

(ii)~ $\underline{\chi}(G_1,\bullet)\simeq\underline{\chi}(G_2,\bullet)$ as functors $\mathbf{Set}^{\mathbf{inj}}\to\mathbf{Set}^{\mathbf{inj}}$.
\end{thm}
The key ingredient of the proof is the definition of stable partitions of graphs, which we now describe.
\begin{defn}[Stable partition]
A stable partition $\Pi=\{\pi_{i}\neq\varnothing\mid i\in I\}$ of a graph $G$ is a partition of the vertex set $V$ such that if vertices $v$ and $u$ are in the same $\pi_{i}$, then $\{v,u\}$ is not an edge of $G$.
\end{defn}
Denote by $\mathbf{St}(G)$ the set of all stable partitions of $G$ and $\mathbf{St}_{\kappa}(G)$ the set of stable partitions of $G$ with cardinality $\kappa$. Then
\[
\mathbf{St}(G)=\bigsqcup_{\kappa}\mathbf{St}_{\kappa}(G),
\]
where $\kappa$ runs over all cardinal numbers. Stable partitions are closely related to graph colorings. Given a regular coloring $c\in\underline{\chi}(G,S)$, there is a stable partition $\Pi_c$ associated to $c$ defined as
\[
\Pi_c:=\{c^{-1}(s)\mid s\in S,c^{-1}(s)\neq\varnothing\}.
\]
Every stable partition $\Pi$ arises in this way, just let $c:V\to\Pi$ be the map taking each vertex to the member of $\Pi$ containing it. Then $\Pi=\Pi_c$. For $\Pi_c$ associated to $c\in\underline{\chi}(G,S)$, the map $\tilde{c}:\Pi_c\to S$ mapping $c^{-1}(s)$ to $s$ is injective and can be regarded as an $S$-coloring of the complete graph $K_{\Pi_c}$ with vertex set $\Pi_c$. The map $c\mapsto\tilde{c}$ in fact defines a natural isomorphism as in the next proposition, which reveals the significant relation between chromatic functors and stable partitions. It is also the key to the proof of Theorem \ref{Yos}.
\begin{prop}[\cite{Yoshinaga2015}]\label{dec}
The map
\begin{align*}
\Omega^G_S:\underline{\chi}(G,S)&\to\bigsqcup_{\Pi\in\mathbf{St}(G)}\underline{\chi}(K_{\Pi},S)\\
c&\mapsto \tilde{c}\in\underline{\chi}(K_{\Pi_c},S)
\end{align*}
has inverse map
\begin{align*}
(\Omega^G_S)^{-1}:\bigsqcup_{\Pi\in\mathbf{St}(G)}\underline{\chi}(K_{\Pi},S)&\to\underline{\chi}(G,S)\\
\underline{\chi}(K_{\Pi_c},S)\ni c^{\prime}&\mapsto c^{\prime}\circ c\in\underline{\chi}(G,S).
\end{align*}
Moreover, $\Omega^G_S$ induces a natural isomorphism of functors
\[
\Omega^G:\underline{\chi}(G,\bullet)\xlongrightarrow{\sim} \bigsqcup_{\Pi\in\mathbf{St}(G)}\underline{\chi}(K_{\Pi},\bullet),
\]
where $K_{\Pi}$ is the complete graph with vertex set $\Pi$.
\end{prop}

In this note, we use this proposition to investigate the following problems.

In Section 2, we determine the group structure of $\mathrm{Aut}(\underline{\chi}(G,\bullet))$, that is the group of natural automorphisms of the chromatic functor $\underline{\chi}(G,\bullet)$. We also compare this group with the graph automorphism group $\mathrm{Aut}(G)$.

In Section 3, we focus on the chromatic functor $\underline{\chi}(G,\bullet):\mathrm{FI}\to\mathrm{FI}$ associated to a finite graph $G$ restricted to the category $\mathrm{FI}$ of finite sets and injections. Consider the composition of functors
\[
\mathrm{FI}\xlongrightarrow{\underline{\chi}(G,\bullet)}\mathrm{FI}\xlongrightarrow{\mathbb{C}[\bullet]}\mathrm{Vect}_{\mathbb{C}},
\]
where $\mathbb{C}[\bullet]$ is the free functor taking a finite set $S$ to the complex vector space spanned by $S$. We show that the consistent sequence $\{\mathbb{C}[\underline{\chi}(G,[n])]\}_n$ of $\mathfrak{S}_n$-representations is representation stable in the sense of Church-Farb \cite{Church2013}.

\section{The automorphism group of chromatic functors}
Let $G$ be a graph. The natural automorphisms of the chromatic functor $\underline{\chi}(G,\bullet):\mathbf{Set}^{\mathbf{inj}}\to\mathbf{Set}^{\mathbf{inj}}$ form a group $\mathrm{Aut}(\underline{\chi}(G,\bullet))$. This group structure is essentially determined by stable partitions of $G$.

Before proving our main result, we recall the uniqueness of stable partitions (Section 2.3 of \cite{Yoshinaga2015}) and give a more explicit expression which implies our result.

Consider two families of sets $\mathfrak{X}=\{X_i\mid i\in I\}$ and $\mathfrak{Y}=\{Y_j\mid j\in J\}$. An isomorphism from $\mathfrak{X}$ to $\mathfrak{Y}$ is a collection $(\alpha;\alpha_i)_{i\in I}$, where $\alpha: I\to J$ and $\alpha_i: X_i\to Y_{\alpha(i)}$ are all bijections of sets. Denote by $\mathrm{Isom}(\mathfrak{X},\mathfrak{Y})$ the set of all isomorphisms from $\mathfrak{X}$ to $\mathfrak{Y}$. Let
\[
F_{\mathfrak{X}}=\bigsqcup_{i\in I}\underline{\chi}(K_{X_i},\bullet),~ F_{\mathfrak{Y}}=\bigsqcup_{j\in J}\underline{\chi}(K_{Y_j},\bullet)
\]
be functors $\mathbf{Set}^{\mathbf{inj}}\to\mathbf{Set}^{\mathbf{inj}}$. Denote by $\mathrm{Isom}(F_{\mathfrak{X}},F_{\mathfrak{Y}})$ the set of all natural isomorphisms from $F_{\mathfrak{X}}$ to $F_{\mathfrak{Y}}$. The next proposition is a recollection of results in Section 2.3 of \cite{Yoshinaga2015}. We rewrite the proof for convenience.
\begin{prop}\label{Xi}
There is a bijection $\mathrm{Isom}(\mathfrak{X},\mathfrak{Y})\to\mathrm{Isom}(F_{\mathfrak{X}},F_{\mathfrak{Y}})$.
\end{prop}
\begin{proof}
If there exists an isomorphism $(\alpha;\alpha_i)_{i\in I}\in\mathrm{Isom}(\mathfrak{X},\mathfrak{Y})$, it determines a natural isomorphism $f:F_{\mathfrak{X}}\to F_{\mathfrak{Y}}$ whose component $f_S:F_{\mathfrak{X}}(S)\to F_{\mathfrak{Y}}(S)$ at $S\in\mathbf{Set}^{\mathbf{inj}}$ sends $c\in\underline{\chi}(K_{X_i},S)$ to $c\circ\alpha_i^{-1}\in\underline{\chi}(K_{Y_{\alpha(i)}},S)$.

Now we construct the inverse map. Given a natural isomorphism $f:F_{\mathfrak{X}}\to F_{\mathfrak{Y}}$, the bijection $\alpha:I\to J$ is defined by the following
\[
F_{\mathfrak{X}}(X_i)\supset\underline{\chi}(K_{X_i},X_i)\ni\mathrm{id}_{X_i}\mapsto f_{X_i}(\mathrm{id}_{X_i})\in\underline{\chi}(K_{Y_{\color{red}{\alpha(i)}}},X_i)\subset F_{\mathfrak{Y}}(X_i).
\]
Moreover, the component $f_S$ of $f$ at $S$ maps an injection $c:X_i\to S$, that is a coloring $c\in\underline{\chi}(K_{X_i},S)$, to $c_*(f_{X_i}(\mathrm{id}_{X_i}))\in\underline{\chi}(K_{Y_{\alpha(i)}},S)$ by naturality of $f$. 
\[
\begin{tikzpicture}
  \matrix (m) [matrix of math nodes,row sep=3em,column sep=4em,minimum width=2em]
  {
     \underline{\chi}(K_{X_i},X_i) & \underline{\chi}(K_{Y_{\alpha(i)}},X_i) \\
     \underline{\chi}(K_{X_i},S) & \underline{\chi}(K_{Y_{\alpha(i)}},S) \\};
  \path[-stealth]
    (m-1-1) edge node [left] {$c_*$} (m-2-1)
            edge node [below] {$f_{X_i}$} (m-1-2)
    (m-2-1) edge node [below] {$f_S$} (m-2-2)
    (m-1-2) edge node [right] {$c_*$} (m-2-2);
\end{tikzpicture}
\]
The bijection $\alpha_i:X_i\to Y_{\alpha(i)}$ is defined as
\begin{align*}
f_{Y_{\alpha(i)}}^{-1}:\underline{\chi}(K_{Y_{\alpha(i)}},Y_{\alpha(i)})&\to\underline{\chi}(K_{X_i},Y_{\alpha(i)})\\
\mathrm{id}_{Y_{\alpha(i)}}&\mapsto\color{red}{\alpha_i}:=f_{Y_{\alpha(i)}}^{-1}(\mathrm{id}_{Y_{\alpha(i)}}).
\end{align*}
This $\alpha_i$ is indeed a bijection with inverse map $f_{X_i}(\mathrm{id}_{X_i}):Y_{\alpha(i)}\to X_i$. Thus we obtain an isomorphism $(\alpha;\alpha_i)_{i\in I}\in\mathrm{Isom}(\mathfrak{X},\mathfrak{Y})$.

It is routine to check that the two maps constructed above are inverse to each other.
\end{proof}

For our purpose, write $\mathrm{Aut}(\mathfrak{X})$ for $\mathrm{Isom}(\mathfrak{X},\mathfrak{X})$ and $\mathrm{Aut}(F_{\mathfrak{X}})$ for $\mathrm{Isom}(F_{\mathfrak{X}},F_{\mathfrak{X}})$. They are both groups under compositions. The previous proposition immediately implies the following result.
\begin{cor}\label{isom}
There is a group isomorphism $\mathrm{Aut}(\mathfrak{X})\to\mathrm{Aut}(F_{\mathfrak{X}})$.
\end{cor}
A direct analysis on how an automorphism of $\mathfrak{X}$ works yields the following.
\begin{lem}\label{autlem}
Let $\mathfrak{X}=\{X_i\mid i\in I\}$ be a family of sets with index set $I$ of arbitrary cardinality. Then
\[
\mathrm{Aut}(\mathfrak{X})\cong \prod_{\kappa}(\mathfrak{S}_{\kappa}\wr\mathfrak{S}_{n_{\kappa}}),
\]
where $n_{\kappa}=\#\{i\in I\mid \#X_i=\kappa\}$ and $\kappa$ runs over all cardinal numbers.
\end{lem}
\begin{proof}
We construct a group isomorphism $\mathrm{Aut}(\mathfrak{X})\to\prod_{\kappa}(\mathfrak{S}_{\kappa}\wr\mathfrak{S}_{n_{\kappa}})$ as follow.
Given $(\alpha;\alpha_i)\in\mathrm{Aut}(\mathfrak{X})$, where $\alpha:I\to I$ and $\alpha_i:X_i\to X_{\alpha(i)}$ are all bijections of sets. If we write $I_{\kappa}:=\{i\in I\mid \#X_i=\kappa\}$, then $I=\bigsqcup_{\kappa}I_{\kappa}$ and $n_{\kappa}=\#I_{\kappa}$. Note that $\alpha|_{I_{\kappa}}$ is a permutation of $I_{\kappa}$, that is an element of $\mathfrak{S}_{n_{\kappa}}$. By fixing an arbitrary bijection $[\kappa]\to X_i$ for each $i\in I_{\kappa}$, where $[\kappa]$ is a standard set with cardinality $\kappa$, the bijection $\alpha_i:X_i\to X_{\alpha(i)}$ corresponds to an element of $\mathfrak{S}_{\kappa}$. This construction obviously defines a group isomorphism.
\end{proof}

We are now ready to prove the following theorem.
\begin{thm}\label{aut}
Let $G$ be a (possibly infinite) graph and $\underline{\chi}(G,\bullet):\mathbf{Set}^{\mathbf{inj}}\to\mathbf{Set}^{\mathbf{inj}}$ the chromatic functor associated to $G$. Then
\[
\mathrm{Aut}(\underline{\chi}(G,\bullet))\cong \prod_{\kappa}(\mathfrak{S}_{\kappa}\wr\mathfrak{S}_{n_{\kappa}}),
\]
where $n_{\kappa}$ is the cardinality of $\mathbf{St}_{\kappa}(G)$ and $\kappa$ runs over all cardinal numbers.
\end{thm}

\begin{proof}
By Proposition \ref{dec}, we have the natural isomorphism of functors
\[
\underline{\chi}(G,\bullet)\simeq \bigsqcup_{\Pi\in\mathbf{St}(G)}\underline{\chi}(K_{\Pi},\bullet)=F_{\mathbf{St}(G)},
\]
where we consider $\mathbf{St}(G)$ as a family of sets. Then the theorem follows from Corollary \ref{isom} and Lemma \ref{autlem} with $\mathfrak{X}=\mathbf{St}(G)$,
\[
\mathrm{Aut}(\underline{\chi}(G,\bullet))\xlongrightarrow{\Psi_G}\mathrm{Aut}(F_{\mathbf{St}(G)})\xlongrightarrow{\Xi_G}\mathrm{Aut}(\mathbf{St}(G))\xlongrightarrow{\Theta_G}\prod_{\kappa}(\mathfrak{S}_{\kappa}\wr\mathfrak{S}_{n_{\kappa}}),
\]
where $\Psi_G(F)=\Omega^G\circ F\circ(\Omega^G)^{-1}$ (see Proposition \ref{dec}), $\Xi_G$ is as in Corollary \ref{isom} and $\Theta_G$ is as in Lemma \ref{autlem}. All the three maps are group isomorphisms.
\end{proof}

A graph automorphism $\varphi\in\mathrm{Aut}(G)$ induces a natural automorphism $\overline{\varphi}\in\mathrm{Aut}(\underline{\chi}(G,\bullet))$ as follow. The component of $\overline{\varphi}$ at a set $S$ is just precomposing with $\varphi^{-1}$,
\[
\overline{\varphi}_S(c)=c\circ\varphi^{-1}:V\xlongrightarrow{\varphi^{-1}}V\xlongrightarrow{c}S,
\]
where $c\circ\varphi^{-1}$ is indeed a regular coloring since $\varphi$ maps edges to edges. The naturality follows from the equality
\[
(\sigma\circ c)\circ\varphi^{-1}=\sigma\circ(c\circ\varphi^{-1}),
\]
for any injection $\sigma:S\to T$. The map $\varphi\mapsto\overline{\varphi}$ defines a group homomorphism $\Phi_G:\mathrm{Aut}(G)\to\mathrm{Aut}(\underline{\chi}(G,\bullet))$. Now we have the following sequence of group homomorphisms
\[
\mathrm{Aut}(G)\xlongrightarrow{\Phi_G}\mathrm{Aut}(\underline{\chi}(G,\bullet))\xlongrightarrow{\Psi_G}\mathrm{Aut}(F_{\mathbf{St}(G)})\xlongrightarrow{\Xi_G}\mathrm{Aut}(\mathbf{St}(G))\xlongrightarrow{\Theta_G}\prod_{\kappa}(\mathfrak{S}_{\kappa}\wr\mathfrak{S}_{n_{\kappa}}).
\]

\begin{lem}\label{alpha}
Let $\varphi\in\mathrm{Aut}(G)$ and denote $f:=\Psi_G(\overline{\varphi})\in\mathrm{Aut}(F_{\mathbf{St}(G)})$, $(\alpha;\alpha_{\Pi}):=\Xi_G(f)\in\mathrm{Aut}(\mathbf{St}(G))$, where $\alpha:\mathbf{St}(G)\to\mathbf{St}(G)$ and $\alpha_{\Pi}:\Pi\to\alpha(\Pi)$ are bijections. Then
\[
\alpha(\Pi_c)=\Pi_{c\circ\varphi^{-1}},
\]
for $\Pi_c\in\mathbf{St}(G)$ associated to $c\in\underline{\chi}(G,\Pi_c)$ and $\Pi_{c\circ\varphi^{-1}}\in\mathbf{St}(G)$ associated to $c\circ\varphi^{-1}\in\underline{\chi}(G,\Pi_c)$. Furthermore,
\begin{align*}
\alpha_{\Pi_c}:\Pi_c&\to\Pi_{c\circ\varphi^{-1}}\\
\pi&\mapsto \varphi(\pi),
\end{align*}
where $\varphi(\pi)=\{\varphi(v)\mid v\in\pi\}$.
\end{lem}
\begin{proof}
For $\Pi_c\in\mathbf{St}(G)$ with $c:V\to\Pi_c$, recall from Proposition \ref{Xi} that $\alpha(\Pi_c)$ is defined as
\[
F_{\mathbf{St}(G)}(\Pi_c)\supset\underline{\chi}(K_{\Pi_c},\Pi_c)\ni\mathrm{id}_{\Pi_c}\mapsto f_{\Pi_c}(\mathrm{id}_{\Pi_c})\in\underline{\chi}(K_{\color{red}{\alpha(\Pi_c)}},\Pi_c)\subset F_{\mathbf{St}(G)}(\Pi_c).
\]
By definition of $\Psi_G$, in fact the commutativity of the following diagram
\[
\begin{tikzpicture}
  \matrix (m) [matrix of math nodes,row sep=3em,column sep=4em,minimum width=2em]
  {
     \underline{\chi}(G,\Pi_c) &  F_{\mathbf{St}(G)}(\Pi_c)\\
     \underline{\chi}(G,\Pi_c) & F_{\mathbf{St}(G)}(\Pi_c) \\};
  \path[-stealth]
    (m-1-1) edge node [left] {$\overline{\varphi}_{\Pi_c}$} (m-2-1)
    (m-1-1) edge node [below] {$\Omega^G_{\Pi_c}$} (m-1-2)
    (m-2-1) edge node [below] {$\Omega^G_{\Pi_c}$} (m-2-2)
    (m-1-2) edge node [right] {$f_{\Pi_c}$} (m-2-2);
\end{tikzpicture}
\]
we obtain $f_{\Pi_c}(\mathrm{id}_{\Pi_c})=\widetilde{c\circ\varphi^{-1}}:\Pi_{c\circ\varphi^{-1}}\to\Pi_c$. Hence
\[
\alpha(\Pi_c)=\Pi_{c\circ\varphi^{-1}}.
\]
The bijection $\alpha_{\Pi_c}:\Pi_c\to\Pi_{c\circ\varphi^{-1}}$ is the inverse of $\widetilde{c\circ\varphi^{-1}}$. By definition of $\widetilde{c\circ\varphi^{-1}}$ (see Section 1), we have the commutative diagram
\[
\begin{tikzpicture}
  \matrix (m) [matrix of math nodes,row sep=3em,column sep=4em,minimum width=2em]
  {
     V &  V\\
     \Pi_{c\circ\varphi^{-1}} & \Pi_c \\};
  \path[-stealth]
    (m-1-1) edge node [left] {proj} (m-2-1)
    (m-1-1) edge node [below] {$\varphi^{-1}$} (m-1-2)
    (m-2-1) edge node [below] {$\widetilde{c\circ\varphi^{-1}}$} (m-2-2)
    (m-1-2) edge node [right] {$c$} (m-2-2);
\end{tikzpicture}
\]
If $v\in V$ such that $c(v)=\pi\in\Pi_c$, then $\varphi(v)\in\varphi(\pi)$. Therefore
\[
\alpha_{\Pi_c}(\pi)=\varphi(\pi). \qedhere
\]
\end{proof}

\begin{prop}
Let $G$ be a (possibly infinite) graph. Then the group homomorphism $\Phi_G:\mathrm{Aut}(G)\to\mathrm{Aut}(\underline{\chi}(G,\bullet))$ is always injective. It is an isomorphism if and only if $G$ is a complete graph or the graph with $2$ vertices and no edges.
\end{prop}
\begin{proof}
If $\varphi,\psi\in\mathrm{Aut}(G)$ have the same image $\overline{\varphi}=\overline{\psi}\in\mathrm{Aut}(\underline{\chi}(G,\bullet))$, that is, for any $c\in\underline{\chi}(G,S)$,
\[
c\circ\varphi^{-1}=c\circ\psi^{-1}.
\]
Then we have $c\circ\varphi^{-1}(v)=c\circ\psi^{-1}(v)$ for any $v\in V$. In particular, let $c$ be injective, we conclude $\varphi=\psi$. This shows that $\Phi_G$ is injective.

Now we suppose that $G$ is the complete graph $K_V$. It is known that $\mathrm{Aut}(G)=\mathfrak{S}_V\cong\mathfrak{S}_{\#V}$. Since $G$ has a unique stable partition $\widehat{V}=\{\{v\}\mid v\in V\}$, by Theorem \ref{aut}, $\mathrm{Aut}(\underline{\chi}(G,\bullet))\cong\mathfrak{S}_{\#V}$. Lemma \ref{alpha} implies that the injective homomorphism $\Phi_G$ must be an isomorphism. If $G$ is the graph with $2$ vertices and no edges, it is easy to see that $\mathrm{Aut}(G)\cong\mathrm{Aut}(\underline{\chi}(G,\bullet))\cong\mathfrak{S}_2$. The same conclusion follows.

Conversely, first we suppose that $G$ is a finite but not complete graph nor the graph with $2$ vertices and no edges. Say $n=\#V$, then $n\geq 3$. There exists a pair of vertices $(v,u)$ such that $\{v,u\}$ is not an edge. Then $G$ may lose some symmetries of $K_n$, in other words,
\[
|\mathrm{Aut}(G)|\leq|\mathfrak{S}_n|=n!.
\]
On the other hand, $G$ has fewer constrains of colorings than $K_n$. In terms of stable partitions, $\mathbf{St}_{n-1}(G)$ is not empty since $v$ and $u$ can be colored with the same color. Then by Theorem \ref{aut},
\[
|\mathrm{Aut}(\underline{\chi}(G,\bullet))|\geq |\mathfrak{S}_{n-1}\wr\mathfrak{S}_{n_{n-1}}|\cdot|\mathfrak{S}_n\wr\mathfrak{S}_1|=(n-1)!^{n_{n-1}}n_{n-1}!n!,
\]
where $n_{n-1}=\#\mathbf{St}_{n-1}(G)\geq 1$. Then
\[
|\mathrm{Aut}(G)|\leq n!< (n-1)!^{n_{n-1}}n_{n-1}!n!\leq |\mathrm{Aut}(\underline{\chi}(G,\bullet))|,
\]
since $n\geq 3$. Thus $\Phi_G$ can not be surjective.

Secondly, if $G$ is an infinite but not complete graph. There is a pair of vertices $(v,u)$ such that $\{v,u\}$ is not an edge. In this case, there are at least $2$ stable partitions of cardinality $\#V$, one is $\widehat{V}=\{\{w\}\mid w\in V\}$, the other is
\[
\Pi(v,u)=\{\{v,u\},\{w\}\mid w\in V\backslash\{v,u\}\}.
\]
If $(\alpha;\alpha_{\Pi})=\Xi_G(\Psi_G(\overline{\varphi}))\in\mathrm{Aut}(\mathbf{St}(G))$ for some $\varphi\in\mathrm{Aut}(G)$, then by Lemma \ref{alpha}, $\alpha_{\widehat{V}}$ maps each $\{w\}\in\widehat{V}$ to $\{\varphi(w)\}\in\widehat{V}$. This asserts that $\alpha(\widehat{V})=\widehat{V}$. Thus $(\beta;\beta_{\Pi})\in\mathrm{Aut}(\mathbf{St}(G))$ with $\beta(\widehat{V})=\Pi(v,u)$ is not in the image of $\Xi_G\circ\Psi_G\circ\Phi_G$. Therefore $\Phi_G$ can not be surjective. 

This completes the proof.
\end{proof}

\section{Representation stability}
Representation stability is a phenomenon arising in various branches of mathematics. Loosely speaking, a sequence of representations $V_n$ of a family of groups $G_n$ is representation stable if the growth of the irreducible decomposition of $V_n$ with respect to $n$ ``stabilizes'' in some sense. Since Church and Farb first introduced the idea in the early version of \cite{Church2013}, the subject has gained much attention. 

In this section, after a brief review of the definition of representation stability for $\mathfrak{S}_n$-representations, we present a new example of representation stability concerning chromatic functors of graphs. We refer the readers to their original paper \cite{Church2013} for details of this subject or to \cite{Farb2014} for a shorter survey. For basic facts of combinatorics and representation theory of symmetric groups, we refer to \cite{Stanley1999}.

\begin{defn}[Consistent sequence of $\mathfrak{S}_n$-representations]
A sequence $\{V_n,\phi_n\}_{n\geq 1}$ of finite dimensional complex representations $V_n$ of the symmetric groups $\mathfrak{S}_n$ together with linear transformations $\phi_n:V_n\to V_{n+1}$ is said to be consistent if the following diagram commutes
\[
\begin{tikzpicture}
  \matrix (m) [matrix of math nodes,row sep=3em,column sep=4em,minimum width=2em]
  {
     V_n & V_{n+1} \\
     V_n & V_{n+1} \\};
  \path[-stealth]
    (m-1-1) edge node [left] {$\sigma$} (m-2-1)
            edge node [below] {$\phi_n$} (m-1-2)
    (m-2-1) edge node [below] {$\phi_n$} (m-2-2)
    (m-1-2) edge node [right] {$\sigma$} (m-2-2);
\end{tikzpicture}
\]
for $n\geq 1$ and $\sigma\in\mathfrak{S}_n$, where $\sigma$ on the right is considered as an element of $\mathfrak{S}_{n+1}$ via the usual group inclusion $\mathfrak{S}_n\hookrightarrow\mathfrak{S}_{n+1}$.
\end{defn}

Irreducible (complex) representations of $\mathfrak{S}_n$ are classified by partitions $\{\mu\mid \mu\vdash n\}$ of the positive integer $n$. The latter could also be identified with Young diagrams with $n$ boxes. For a partition (Young diagram) $\mu\vdash n$, denote by $V_{\mu}$ the irreducible $\mathfrak{S}_n$-representation associated to $\mu$. In the context of representation stability, for a partition $\lambda=(\lambda_1,\ldots,\lambda_l)\vdash m$ and $n\geq m+\lambda_1$, let
\[
\lambda[n]:=(n-m,\lambda_1,\ldots,\lambda_l)\vdash n.
\]
We prefer to use
\[
V(\lambda)_n:=V_{\lambda[n]}
\]
to denote the associated irreducible $\mathfrak{S}_n$-representation.

\begin{defn}[Representation stability for $\mathfrak{S}_n$-representations]
Let $\{V_n,\phi_n\}_n$ be a consistent sequence of $\mathfrak{S}_n$-representations. We say that the sequence $\{V_n,\phi_n\}_n$ is representation stable with stable range $n\geq N$ if there is a positive integer $N$ such that the following conditions hold for $n\geq N$.
\begin{itemize}
\item (Injectivity) $\phi_n$ is injective.
\item (Orbit-surjectivity) The $\mathfrak{S}_{n+1}$-orbit of the image of $\phi_n$ equals $V_{n+1}$.
\item (Multiplicity stability) In the irreducible decomposition of $V_n$
\[
V_n=\bigoplus_{\lambda}V(\lambda)_n^{\oplus c_{n}(\lambda)},
\]
the multiplicity $0\leq c_n(\lambda)<\infty$ is independent of $n$.
\end{itemize}
\end{defn}

Known examples of consistent sequences that are representation stable includes
\begin{itemize}
\item $\{H^i(\mathrm{Conf}_n(M);\mathbb{C})\}_n$. The complex cohomology of the configuration space of $n$ distinct points on a connected orientable manifold $M$ of finite type (\cite{Church2012}).
\item $\{H^i(\mathcal{M}_g^n;\mathbb{C})\}_n$. The complex cohomology of the moduli space of Riemann surfaces of genus $g\geq 2$ with $n$ marked points (\cite{Rolland2011}).
\end{itemize}
See \cite{Church2015} for more examples. 

Our main result of this section is a new example of representation stability. Let $G=(V,E)$ be a finite graph, consider the following composition of functors
\[
\mathrm{FI}\xlongrightarrow{\underline{\chi}(G,\bullet)}\mathrm{FI}\xlongrightarrow{\mathbb{C}[\bullet]}\mathrm{Vect}_{\mathbb{C}},
\]
where $\mathrm{Vect}_{\mathbb{C}}$ is the category of complex vector spaces and $\mathbb{C}[\bullet]$ is the free functor taking a finite set $S$ to the complex vector space spanned by $S$. For a positive integer $n$, the vector space $\mathbb{C}[\underline{\chi}(G,[n])]$ has an $\mathfrak{S}_n$-representation structure as the permutation representation of the natural action of $\mathfrak{S}_n$ on $\underline{\chi}(G,[n])$ permuting colors. Explicitly, if $c\in\underline{\chi}(G,[n])$ is a (regular) $n$-coloring of $G$ and $\sigma\in\mathfrak{S}_n$ a permutation, then $\sigma c$ is the $n$-coloring $V\xlongrightarrow{c}[n]\xlongrightarrow{\sigma}[n]$. 

Let $\phi_n:\mathbb{C}[\underline{\chi}(G,[n])]\to\mathbb{C}[\underline{\chi}(G,[n+1])]$ be the linear map sending a basis element $c$ to $\iota_n\circ c:V\to [n]\hookrightarrow [n+1]$, where $\iota_n:[n]\hookrightarrow [n+1]$ is the usual inclusion. The sequence $\{\mathbb{C}[\underline{\chi}(G,[n])]\}_{n\geq1}$ is a consistent sequence of $\mathfrak{S}_n$-representations.

\begin{thm}\label{mainthm}
Let $G=(V,E)$ be a finite graph. Then the sequence $\{\mathbb{C}[\underline{\chi}(G,[n])]\}_n$ is representation stable for $n\geq 2\#V$.
\end{thm}
 
Our strategy of the proof is to use Proposition \ref{dec} to reduce the problem to the case of complete graphs. By Proposition \ref{dec}, we have the following isomorphism of functors
\[
\underline{\chi}(G,\bullet)\simeq \bigsqcup_{1\leq k\leq \#V}\underline{\chi}(K_k,\bullet)^{\sqcup n_k}
\]
where $K_k$ denotes the complete graph on $k$ vertices and $n_k=\#\mathbf{St}_k(G)$. This yields the following isomorphism of functors
\begin{equation}
\mathbb{C}[\underline{\chi}(G,\bullet)]\simeq \bigoplus_{1\leq k\leq \#V}\mathbb{C}[\underline{\chi}(K_k,\bullet)]^{\oplus n_k}. \label{decomp}
\end{equation}
The question now reduces to the case of complete graphs. Note that the $\mathfrak{S}_n$-representation $\mathbb{C}[\underline{\chi}(K_k,[n])]$ is obviously isomorphic to the permutation representation of the $\mathfrak{S}_n$-action on the set
\[
\mathrm{Conf}_k([n]):=\{(i_1,\ldots,i_k)\mid i_j\in[n], i_j\neq i_{j^{\prime}} \text{ if } j\neq j^{\prime}\}
\]
given by $\sigma(i_1,\ldots,i_k)=(\sigma(i_1),\ldots,\sigma(i_k))$ for $\sigma\in\mathfrak{S}_n$.

Recall that a semistandard Young tableaux (SSYT) is an assignment of one positive integer to each box of a Young diagram (called its shape) such that each row is non-strictly increasing and each column is strictly increasing. The type of an SSYT is the tuple $(m_1,m_2,\ldots)$ where $m_i$ is the times number $i$ appears in this SSYT. For a Young diagram $\mu$ and a tuple $\alpha$ of nonnegative integers, the Kostka number $K_{\mu,\alpha}$ is the number of SSYTs of shape $\mu$ and of type $\alpha$ (see Section 7.10 of \cite{Stanley1999}).

\begin{lem}\label{lem1}
For $n\geq k$, the irreducible decomposition of the permutation representation $\mathbb{C}[\mathrm{Conf}_k([n])]$ is
\[
\mathbb{C}[\mathrm{Conf}_k([n])]=\bigoplus_{\mu\vdash n}V_{\mu}^{\oplus K_{\mu,\alpha(n,k)}},
\]
where $\alpha(n,k)=(n-k,\overbrace{1,\ldots,1}^{k})$.
\end{lem}
\begin{proof}
Consider the Young subgroup $\mathfrak{S}_{\alpha(n,k)}$ of $\mathfrak{S}_n$ associated to the partition $\alpha(n,k)\vdash n$. That is the subgroup consisting of permutations that permute the first $n-k$ numbers and fix the others. Then $\mathfrak{S}_{\alpha(n,k)}$ is the stabilizer of $(n-k+1,\ldots,n)\in\mathrm{Conf}_k([n])$. One observes that
\[
\mathbb{C}[\mathrm{Conf}_k([n])]=\mathrm{Ind}_{\mathfrak{S}_{\alpha(n,k)}}^{\mathfrak{S}_n}\mathbb{C},
\]
where the right hand side is the induced representation of the trivial $\mathfrak{S}_{\alpha(n,k)}$-representation. Then Young's rule (Proposition 7.18.7 of \cite{Stanley1999}) gives
\[
\mathrm{Ind}_{\mathfrak{S}_{\alpha(n,k)}}^{\mathfrak{S}_n}\mathbb{C}=\bigoplus_{\mu\vdash n}V_{\mu}^{\oplus K_{\mu,\alpha(n,k)}}. \qedhere
\]
\end{proof}

\begin{thm}\label{main}
The sequence $\{\mathbb{C}[\underline{\chi}(K_k,[n])]\}_n$ is representation stable for $n\geq2k$.
\end{thm}
\begin{proof}
It is easy to check that $\phi_n:\mathbb{C}[\underline{\chi}(K_k,[n])]\to\mathbb{C}[\underline{\chi}(K_k,[n+1])]$ is injective and that the $\mathfrak{S}_{n+1}$-orbit of the image of $\phi_n$ equals $\mathbb{C}[\underline{\chi}(K_k,[n+1])]$ for $n\geq 2k$. It suffices to show that in the irreducible decomposition
\[
\mathbb{C}[\underline{\chi}(K_k,[n])]=\bigoplus_{\lambda}V(\lambda)_n^{\oplus c_n(\lambda)},
\]
the multiplicity $c_n(\lambda)$ is independent of $n$ if $n\geq 2k$. By Lemma \ref{lem1}, for fixed $k$ and $n\geq k$,
\[
\mathbb{C}[\underline{\chi}(K_k,[n])]=\bigoplus_{\mu\vdash n}V_{\mu}^{\oplus K_{\mu,\alpha(n,k)}}.
\]
Consider an SSYT $T$ of type $\alpha(n,k)$, the $1$'s have to occupy the first $n-k$ (or $1$ if $n=k$) boxes in the first row. We only need to arrange the $k$ boxes labeled by $2,\ldots,k+1$. First we have
\[
\begin{ytableau}
1 & \cdots & 1 & 2 & 3 & \cdots & \scriptstyle k+1
\end{ytableau}
\]
which contributes to one copy of the trivial representation $V(0)_n$. We can also move $m$ of the $k$ boxes labeled by $2,\ldots,k+1$ to lower rows and form a $\lambda\vdash m$, which contributes to one copy of $V(\lambda)_n$. Therefore the multiplicity
\[
c_n(\lambda)=K_{\lambda[n],\alpha(n,k)}.
\]
It is remarkable to note that for a fixed $\lambda$, this number is constant once $n$ is large enough since the only change occurs in the first row when $n$ increases. The extreme case is
\[
\begin{ytableau}
1 & \cdots & 1 & \cdots & 1\\
2 & \cdots & \scriptstyle k+1 
\end{ytableau}
\]
which contributes to one copy of $V(k)_n$, and the appearance of this copy requires $n-k\geq k$. This gives the stable range $n\geq 2k$.
\end{proof}

\begin{proof}[Proof of Theorem \ref{mainthm}]
The injectivity and orbit surjectivity of $\phi_n$ are easily verified. By the isomorphism (\ref{decomp}), we obtain
\begin{align*}
\mathbb{C}[\underline{\chi}(G,[n])]&\simeq \bigoplus_{1\leq k\leq \#V}\mathbb{C}[\underline{\chi}(K_k,[n])]^{\oplus n_k}\\
&\simeq \bigoplus_{1\leq k\leq \#V}\left(\bigoplus_{\lambda[n]\vdash n}V(\lambda)_n^{\oplus K_{\lambda[n],\alpha(n,k)}}\right)^{\oplus n_k}\\
&=\bigoplus_{\lambda[n]\vdash n}V(\lambda)_n^{\oplus\sum_{1\leq k\leq \#V}n_k K_{\lambda[n],\alpha(n,k)}}
\end{align*}
Theorem \ref{main} asserts that for a fixed $\lambda$, the number $K_{\lambda[n],\alpha(n,k)}$ is constant for $n\geq 2k$. Therefore for a fixed $\lambda$, the multiplicity
\[
\sum_{1\leq k\leq \#V}n_k K_{\lambda[n],\alpha(n,k)}
\]
is constant for $n\geq 2\#V$.

This shows that the sequence $\{\mathbb{C}[\underline{\chi}(G,[n])]\}_n$ is representation stable for $n\geq 2\#V$.
\end{proof}

\proof[Acknowledgements]
The author would like to thank Professor Masahiko Yoshinaga for posing the questions and for valuable discussions.

\end{document}